\documentclass{amsart}

\usepackage{amssymb}
\usepackage{amsmath}
\usepackage{amsfonts}
\usepackage{amstext}
\usepackage{amsthm}
\usepackage{microtype}

\usepackage{amscd}

\usepackage{ae,aecompl}
\usepackage[T1]{fontenc}

\usepackage{color}                    
\usepackage{hyperref}

\newtheorem{theorem}{Theorem}[section]
\newtheorem{lemma}[theorem]{Lemma}
\newtheorem{proposition}[theorem]{Proposition}
\newtheorem{corollary}[theorem]{Corollary}

\theoremstyle{definition}
\newtheorem{remark}[theorem]{Remark}
\newtheorem{definition}[theorem]{Definition}
\newtheorem{example}[theorem]{Example}

\numberwithin{equation}{section}

\DeclareMathOperator{\Alg}{Alg}

\newcommand{\Sc}[1]{\mathcal{#1}}
\newcommand{\F}[1]{\mathfrak{#1}}
\newcommand{\B}[1]{\mathbb{#1}}

\newcommand{\isocross}{ \Sc{A} \sideset{_{N}}{_\alpha^{iso}}{\mathop{\times}} \Sc{S}}
\newcommand{\cross}{ \Sc{A} \sideset{_{N}}{_\alpha}{\mathop{\times}} \Sc{S}}
\newcommand{\fullcross}{\Sc{B}\times_{\beta}\Sc{G}}
\newcommand{\relcross}{\Sc{A}\sideset{_{N}}{_{\Sc{C},\alpha}}{\mathop{\times}} \Sc{S}}

\newcommand{\<}{\langle}
\renewcommand{\>}{\rangle}

\title[Semicrossed Products]{Nonself-adjoint semicrossed products by abelian semigroups}
\author{Adam Hanley Fuller}
\address{Pure Math. Dept., U. Waterloo, Waterloo, ON N2L-3G1, CANADA}
\email{a2fuller@math.uwaterloo.ca}
\begin{document}

\begin{abstract}
	Let $\mathcal{S}$ be the semigroup $\mathcal{S}=\sum^{\oplus k}_{i=1}\Sc{S}_i$, where for each $i\in I$, 
	$\mathcal{S}_i$ is a countable subsemigroup of the additive semigroup $\B{R}_+$ containing $0$. We consider representations
	of $\mathcal{S}$ as contractions $\{T_s\}_{s\in\mathcal{S}}$ on a Hilbert space with the Nica-covariance property:
	$T_s^*T_t=T_tT_s^*$ whenever $t\wedge s=0$. We show that all such representations have a unique minimal isometric Nica-covariant
	dilation.

	This result is used to help analyse the nonself-adjoint semicrossed product algebras formed from Nica-covariant 		representations of the action of $\mathcal{S}$ on an operator algebra $\mathcal{A}$ by completely contractive 		endomorphisms.
	We conclude by calculating the $C^*$-envelope of the isometric nonself-adjoint semicrossed product algebra (in the sense
	of Kakariadis and Katsoulis).
\end{abstract}

\maketitle
\section{Introduction}
The study of nonself-adjoint semicrossed products began with Arveson \cite{Arveson1}. They were further studied by McAsey, Muhly
and Saito \cite{McAseyMuhly}. In both cases the algebras were described concretely. Peters \cite{peters1} described the nonself-adjoint semicrossed products as universal algebras for covariant representations. In recent years, Davidson and Katsoulis have shown nonself-adjoint semicrossed products have proven to be a particularly interesting and tractable class of operator algebras 
\cite{DavKat0,DavKat3,DavKat4,DavKat6,DavKat1}. In particular, nonself-adjoint semicrossed product algebras have been shown to be a class where the $C^*$-envelope is often calculable.

The $C^*$-envelope of an operator algebra $\Sc{A}$ was introduced by Arveson \cite{Arveson2,Arveson3} as a non-commutative analogue of Shilov boundaries. The existence of the $C^*$-envelope was first discovered by Hamana \cite{Hamana}. Dritschel and McCullough 
\cite{DritMcCul} have since provided an alternative proof of the existence of the $C^*$-envelope. The viewpoint of Dritschel
and McCullough has allowed for the explicit calculation of the $C^*$-envelope of many operator algebras. In particular, for nonself-adjoint semicrossed products the $C^*$-envelopes have been studied in
\cite{DavKat3,kak, kakkat,Duncan,DuncanPeters}.

In this paper we study the nonself-adjoint semicrossed product algebras by semigroups of the form
$\Sc{S}=\sum^{\oplus k}_{i=1}\Sc{S}_i$, where for each $i\in I$ we have $\Sc{S}_i$ is a countable subsemigroup of the
additive semigroup $\B{R}_+$ containing $0$. Our algebras will be universal for Nica-covariant covariant representations, i.e. those
representations $\{T_s\}_{s\in\Sc{S}}$ satisfying $T_s^*T_t=T_tT_s^*$ when $s\wedge t=0$. Semicrossed product algebras associated
to Nica-covariant representations have been widely studied in the $C^*$-algebra literature \cite{Nica,LacaRaeburn,Fowler}.

The paper is divided into three sections. In section 2 Nica-covariant representations are studied independent from dynamical systems. The results of this section may be of interest, even to those not concerned with nonself-adjoint semicrossed products. We show that contractive Nica-covariant representations can be dilated to isometric Nica-covariant representations. This result is well-known for the case of the semigroup $\B{Z}_+^k$, see e.g. \cite{Timotin}. The proof of the existence of an isometric dilation presented here relies on the use of a generalisation of the Schur Product Theorem, and so provides an alternative proof to what is usually presented for $\B{Z}_+^k$.

In section 3 nonself-adjoint semicrossed product algebras are introduced. In section 3.1 we extend our dilation result from 
section 2 to representations of semicrossed products of $C^*$-algebras. This result allows us to conclude strong results comparing the different types of semicrossed product algebras. For example, Corollary \ref{same alg}, tells us that, in the case of a semicrossed product of a $C^*$-algebra, the universal algebra for \emph{completely isometric} Nica-covariant representations is the same as the universal algebra for \emph{completely contractive} Nica-covariant representations. If  we were to work with completely isometric and completely contractive semicrossed algebras without imposing the condition of Nica-covariance on  our semigroup representations, then an example due to Varopoulos \cite{Var} would show that the analogy of Corollary \ref{same alg} would fail in this setting.

In the section 3.2 we consider the $C^*$-envelope of the isometric semicrossed product algebras. In Theorem \ref{c* env} we calculate the $C^*$-envelope of the isometric semicrossed product as
	\begin{equation*}
		C^*_{env}(\isocross)\cong C^*_{env}(\Sc{A})\times_\alpha \Sc{G},
	\end{equation*}
where $\Sc{G}$ is the group generated by $\Sc{S}$.
This result generalises a recent result of Kakariadis and Katsoulis
\cite{kakkat}, where they worked with the semigroup $\Sc{S}=\B{Z}_+$.

When $\Sc{A}$ is a $C^*$-algebra the Nica-covariance requirement on our representations allows us to view the semicrossed product algebra $\cross$ as a tensor algebra for a product system of $C^*$-correspondences over $\Sc{S}$. Thus, from this viewpoint we unite a recent result of Duncan and Peters \cite{DuncanPeters} on the
$C^*$-envelope of a tensor algebra associated with a dynamical system and the results of Kakariadis and Katsoulis on the $C^*$-envelope of the isometric semicrossed product for a dynamical system.

\section{Nica-covariant representations of abelian semigroups}
Let $\Sc{S}$ be the semigroup $\Sc{S}=\sum^\oplus_{i\in I}\Sc{S}_i$, where for each $i\in I$ we have $\Sc{S}_i$ is a subsemigroup in the additive semigroup $\B{R}_+$ containing $0$. We further assume throughout that $\Sc{S}$ is the positive cone of the group $\Sc{G}$ it generates. Denote by $\wedge$ and $\vee$ the join and meet operations on the lattice group $\Sc{G}$. In section 3 we will be looking at the case when $\Sc{S}$ is countable. However, we will not need to assume that $\Sc{S}$ is countable in this section.

\begin{definition}
	A representation $T:\Sc{S}\rightarrow\Sc{B}(\Sc{H})$ of $\Sc{S}$ by contractions $\{T_s\}_{s\in\Sc{S}}$ on a Hilbert space $\Sc{H}$ is \emph{Nica-covariant} when
	we have the following relation: if $s\in\Sc{S}_i$ and $t\in\Sc{S}_j$ where $i\neq j$ then $T_s^*T_t=T_tT_s^*$.
\end{definition}

\subsection{Isometric Dilations}
We wish to show that every Nica-covariant contractive representation of $\Sc{S}$ can be dilated to an isometric representation. Further, we will show that there is a unique
minimal isometric dilation which is Nica-covariant. This result is well known in the discrete $\Sc{S}=\B{Z}_+^k$ case.  If each $\Sc{S}_i$ is \emph{commensurable}, i.e. if for all $s_1,\ldots,s_n\in\Sc{S}_i$
there exists $s_0\in\Sc{S}_i$ and $k_1,\ldots,k_n\in\B{N}$ such that $s_i=k_is_0$, then these results have been described by Shalit \cite{Shal1}. We do not impose the
condition of commensurability.

The key method to show the existence of the dilation is to use a generalisation of the Schur Product Theorem. To show that there is a minimal Nica-covariant isometric dilation
we follow arguments similar to those of Solel \cite{Solel}.

\begin{definition}
	Let $A=[A_{i,j}]_{1\leq i,j\leq m}$ and $B=[B_{i,j}]_{1\leq i,j\leq m}$	be two matrices of operators where each $A_{i,j}$ and $B_{i,j}$ is a bounded operator on
	a Hilbert space $\Sc{H}$.
	The \emph{operator-valued Schur product} of $A$ and $B$ is defined by $A\square B:=[A_{i,j}B_{i,j}]_{1\leq i,j\leq m}$. 
\end{definition}

In the above definition, if $\Sc{H}$ is $1$-dimensional then the operation $\square$ is simply the classical Schur product (or entry-wise product). In the following theorem
we will generalise the Schur Product Theorem, which says that the Schur product of two positive matrices is positive. See e.g. \cite[Chapter 3]{Paul}.

\begin{theorem}\label{schur prod theorem}
	Let $\Sc{A}$ and $\Sc{B}$ be two $C^*$-algebras in $\Sc{B}(\Sc{H})$ such that $\Sc{A}\subseteq\Sc{B}'$. Let $A=[A_{i,j}]_{1\leq i,j\leq m}$
	and $B=[B_{i,j}]_{1\leq i,j\leq m}$ be operator matrices 
	with all $A_{i,j}\in\Sc{A}$ and $B_{i,j}\in\Sc{B}$. If $A\geq0$ and $B\geq0$ then $A\square B\geq0$.
\end{theorem}

\begin{proof}
	Let $\tilde{A}=A\otimes I_m$ and $\tilde{B}=[B_{i,j}\otimes I_m]_{1\leq i,j\leq m}$. Hence $\tilde{A}$ and $\tilde{B}$ are of the form
	\begin{equation*}
		\tilde{A}=\begin{bmatrix}
			A&0&\ldots&0\\
			0&A&\ldots&0\\
			\vdots&&\ddots&\vdots\\
			0&0&\ldots&A
		\end{bmatrix}
	\end{equation*}
	and
	\begin{equation*}
		\tilde{B}=\begin{bmatrix}
			B_{1,1}&0&\ldots&0&\ldots&B_{1,m}&0&\ldots&0\\
			0&B_{1,1}&\ldots&0&\ldots&0&B_{1,m}&\ldots&0\\
			\vdots&&\ddots&\vdots&&\vdots&&\ddots&\vdots\\
			0&0&\ldots&B_{1,1}&\ldots&0&0&\ldots&B_{1,m}\\
			\vdots&&&&\vdots&&&&\vdots\\
			B_{m,1}&0&\ldots&0&\ldots&B_{m,m}&0&\ldots&0\\
			0&B_{m,1}&\ldots&0&\ldots&0&B_{m,m}&\ldots&0\\
			\vdots&&\ddots&\vdots&&\vdots&&\ddots&\vdots\\
			0&0&\ldots&B_{m,1}&\ldots&0&0&\ldots&B_{m,m}
		\end{bmatrix}
	\end{equation*}
	It follows that $\tilde{A}$ and $\tilde{B}$ are positive commuting operators.
	Hence $\tilde{A}\tilde{B}$ is positive.

	For each $1\leq k\leq m$, let $P_k$ be the projection onto the $(m(k-1)+k)^{th}$ copy of $\Sc{H}$ in $\Sc{H}^{(m^2)}$, and let $P=\sum_{k=1}^m P_k$.
	Define $R:\Sc{H}^{(m)}\rightarrow\Sc{H}^{(m^2)}$ by $R\mathbf{h}=P(\mathbf{h}^{\otimes m})$. Hence $R$ is an isometry and for
	\begin{equation*}
		\mathbf{h}=\begin{bmatrix}h_1\\h_2\\ \vdots\\ h_m\end{bmatrix}\in\Sc{H}^{(m)}
	\end{equation*}
	we have
	\begin{equation*}
		R\mathbf{h}=\begin{bmatrix}h_1\\0\\ \vdots\\0\\h_2\\0\\ \vdots\\0\\ h_m\end{bmatrix},
	\end{equation*}
	with $m$ zeroes between $h_i$ and $h_{i+1}$, $1\leq i<m$.
	It follows that $R^*(\tilde{A}\tilde{B})R=A\square B$.
	Thus, $A\square B$ is positive.
\end{proof}

Let $T$ be a Nica-covariant contractive representation of $\Sc{S}$ on $\Sc{H}$. We extend $T$ to a map on all of $\Sc{G}$ in the following way. Any element $g\in\Sc{G}$ can be
written uniquely as $g=g_+-g_{-}$ where $g_-,g_+\in\Sc{S}$ and $g_-\wedge g_+=0$. Thus we extend $T$ to $\Sc{G}$ by setting $T_g=T_{g_-}^*T_{g_+}=T_{g_+}T_{g_-}^*$. A well-known theorem 
of Sz.-Nagy says that $T$ has an isometric dilation if and only if for $s_1,\ldots,s_n\in\Sc{S}$ the operator matrix $[T_{s_j-s_i}]_{1\leq i,j\leq n}$ is positive (see e.g. 
\cite[Theorem 7.1]{Nagy}). We will need to look more closely at the proof of this later.

In the case when is $\Sc{S}$ is a subsemigroup of $\B{R}_+$ it has been proved by Mlak \cite{Mlak} that a contractive representation $T$ has an isometric dilation. In the following theorem we will rely on the fact that the representation $T$ restricted to $\Sc{S}_i$ has an isometric dilation for each $i$. Then an invocation of Theorem \ref{schur prod theorem} will
give us our result.

\begin{theorem}\label{dilation exists}
	Let $T$ be a Nica-covariant contractive representation of the semigroup $\Sc{S}=\sum^\oplus_{i\in I}\Sc{S}_i$, where each $\Sc{S}_i$ is subsemigroup of $\B{R}_+$ containing $0$.
	Then $T$ has	an isometric dilation.
\end{theorem}

\begin{proof}
	Take $s_1,\ldots,s_n$ in $\Sc{S}$. By \cite[Theorem 7.1]{Nagy} it suffices to show that the operator matrix $[T_{s_j-s_i}]_{1\leq i,j\leq n}$ is positive. Each $s_j$ is of the form
	$s_j=\sum_{i\in I}s_j^{(i)}$, where $s_j^{(i)}$ is in $\Sc{S}_i$. We can choose a finite subset $F\subseteq I$ such that $s_j=\sum_{i\in F}s_j^{(i)}$ for $j=1,\ldots,n$. 
	Since $F$ is finite we can and will relabel $F$ by $\{1,\ldots,k\}$ for some $k$. Denote
	by $T^{(j)}$ the restriction of $T$ to $\Sc{S}_j$.
	
	By the Nica-covariance property
	\begin{equation*}
		T_{s_j-s_i}=T^{(1)}_{s_j^{(1)}-s_i^{(1)}}\ldots T^{(k)}_{s_j^{(k)}-s_i^{(k)}}.
	\end{equation*}
	Thus, we can factor the operator matrix $[T_{s_j-s_i}]$ as
	\begin{align*}
		[T_{s_j-s_i}]_{i,j}
		&=\begin{bmatrix}T^{(1)}_{s_j^{(1)}-s_i^{(1)}}\ldots T^{(k)}_{s_j^{(k)}-s_i^{(k)}}\end{bmatrix}_{i,j}\\
		&=\begin{bmatrix}T^{(1)}_{s_j^{(1)}-s_i^{(1)}}\end{bmatrix}_{i,j}\square\ldots\square\begin{bmatrix}T^{(k)}_{s_j^{(k)}-s_i^{(k)}}\end{bmatrix}_{i,j}.
	\end{align*}
	Since $\big[T^{(l)}_{s_j^{(l)}-s_i^{(l)}}\big]_{i,j}$ is a positive matrix for $1\leq l\leq k$ \cite{Mlak} and since the 	representation is Nica-covariant, it follows by
	Theorem \ref{schur prod theorem} that $[T_{s_j-s_i}]_{i,j}$ is positive.
\end{proof}

In the above we made use of \cite[Theorem 7.1]{Nagy} to guarantee the existence of a dilation. We will now pay closer attention to how the dilation there is constructed. Then,
following similar arguments of \cite{Solel}, we will show that there is a unique minimal Nica-covariant isometric dilation.

\begin{theorem}\label{dilation is Nica-cov}
	Let $T$ be a Nica-covariant contractive representation of the semigroup $\Sc{S}=\sum^\oplus_{i\in I}\Sc{S}_i$, where each $\Sc{S}_i$ is subsemigroup of $\B{R}_+$ containing $0$.
	Then $T$ has a minimal isometric dilation which is Nica-covariant. Further, this dilation is unique.
\end{theorem}

\begin{proof}
	We first sketch the details of the construction of an isometric dilation. Let $\Sc{H}$ be the space on which 		the representation $T$ acts. Let $\Sc{K}_0$ denote the space
	of all finitely non-zero functions $f:\Sc{S}\rightarrow\Sc{H}$. For $f,g\in\Sc{K}_0$ we define
		\begin{equation*}
			\<f,g\>=\sum_{s,t\in\Sc{S}}\<T_{t-s}f(t),g(s)\>.
		\end{equation*}
	By Theorem \ref{dilation exists} this defines a positive semidefinite sesquilinear form on
	$\Sc{K}_0$. Let 
	\begin{align*}
		\Sc{N}&=\{f\in\Sc{K}_0:\<f,f\>=0\}\\ &=\{f\in\Sc{K}_0:\<f,g\>=0\},
	\end{align*}
	and set $\Sc{K}=\overline{\Sc{K}_0/\Sc{N}}$, where the
	closure is taken
	with respect to the norm induced by $\<\cdot,\cdot\>$.
	We isometrically embed $\Sc{H}$ in $\Sc{K}$ by the map $h\mapsto\hat{h}$, where $\hat{h}(s)=\delta_0(s)h$.
	
	Now define maps $V_s$ on $\Sc{K}_0$ by $(V_sf)(t)=f(t-s)$ if $t-s\in\Sc{S}$ and $(V_sf)(t)=0$ otherwise. Note 
	that for $f\in\Sc{K}_0$ and $u\in\Sc{S}$ we have
	\begin{align*}	
		\<V_uf,V_uf\>&=\sum_{s,t}\<T_{t-s}f(t-u),f(s-u)\>\\
			&=\sum_{s,t}\<T_{(t+u)-(s+u)}f(t),f(s)\>\\
			&=\sum_{s,t}\<T_{t-s}f(t),f(s)\>=\<f,f\>.
	\end{align*}
	Hence each $V_u$ is isometric on $\Sc{K}_0$ and leaves $\Sc{N}$ invariant. It follows that we can extend
	$V_u$ to an isometry on $\Sc{K}$ and we have that $\{V_s\}_{s\in\Sc{S}}$ is an isometric representation of $\Sc{S}$.
	
	Further, note that for
	$g\in\Sc{G}$ and $h,k\in\Sc{H}$ we have
	\begin{align*}
		\<V_g\hat{h},\hat{k}\>&=\<V_{g_-}^*V_{g_+}\hat{h},\hat{k}\>=\<V_{g_+}\hat{h},V_{g_-}\hat{k}\>\\
		&=\sum_{s,t\in\Sc{S}}\<T_{t-s}\hat{h}(t-g_+),\hat{k}(s-g_-)\>\\
		&=\<T_g h,k\>.
	\end{align*}
	Thus we have $P_{\Sc{H}}V_g|_{\Sc{H}}=T_g$ for all $g\in\Sc{G}$. In particular $\{V_s\}_{s\in\Sc{S}}$ is an isometric
	dilation of $\{T_s\}_{s\in\Sc{S}}$. It is easily seen to be a minimal isometric dilation.
	Dilation with the property that $P_{\Sc{H}}V_g|_{\Sc{H}}=T_g$ are called a \emph{regular} dilations. We want to
	show that this dilation is Nica-covariant.
	
	Next we will show that if we have $s\in\Sc{S}_i$ and $\mu\in\Sc{S}$ such that $s\wedge\mu=0$ then $V_s^*V_\mu|_\Sc{H}=V_\mu V_s^*|_{\Sc{H}}$. Take $s,\mu$ as described,
	 $\nu\in\Sc{S}$ and $h,k\in\Sc{H}$. By the minimality of the dilation it suffices to show that
	\begin{equation*}
		\<V_s^* V_\mu\hat{h},V_\nu\hat{k}\>=\<V_\mu V_s^*\hat{h},V_\nu\hat{k}\>.
	\end{equation*}
	We calculate
	\begin{align*}
		\<V_s^* V_\mu\hat{h},V_\nu\hat{k}\>&=\<V_\nu^*V_s^* V_\mu\hat{h},\hat{k}\>\\
		&=\<V_{(\mu-\nu-s)_-}^*V_{(\mu-\nu-s)_+}\hat{h},\hat{k}\>\\
		&=\<T_{(\mu-\nu-s)_-}^*T_{(\mu-\nu-s)_+}\hat{h},\hat{k}\>.
	\end{align*}
	Note that, by our choice of $s$ and $\mu$ we have that $(\mu-\nu-s)_+=(\mu-\nu)_+$ and
	$(\mu-\nu-s)_-=s+(\mu-\nu)_-$. Also
	$s\wedge(\mu-\nu)_+=0$. Thus
	\begin{align*}
		\<V_s^* V_\mu\hat{h},V_\nu\hat{k}\>&=\<T_{(\mu-\nu-s)_-}^*T_{(\mu-\nu-s)_+}\hat{h},\hat{k}\>\\
		&=\<T_{s+(\mu-\nu)_-}^*T_{(\mu-\nu)_+}\hat{h},\hat{k}\>\\
		&=\<T_{(\mu-\nu)_-}^*T_{(\mu-\nu)_+}T_s^*\hat{h},\hat{k}\>\\
		&=\<P_{\Sc{H}}V_{(\mu-\nu)_-}^*V_{(\mu-\nu)_+}P_{\Sc{H}}V_s^*\hat{h},\hat{k}\>\\
		&=\<V_{(\mu-\nu)_-}^*V_{(\mu-\nu)_+}V_s^*\hat{h},\hat{k}\>=\<V_\mu V_s^*\hat{h},V_\nu\hat{k}\>.
	\end{align*}
	This tells us that the representation $V$ has the Nica-covariant property when restricted to $\Sc{H}$. We will now extend this to all of $\Sc{K}$.
	
	By the minimality of the representation $V$ it suffices to show that for
	$s\in S_i$, $t\in S_j$ where $i\neq j$, $\mu,\nu\in\Sc{S}$ and $h,k\in\Sc{H}$ that
	\begin{equation*}
		\<V_s^*V_tV_\mu\hat{h},V_\nu\hat{k}\>=\<V_tV_s^*V_\mu\hat{h},V_\nu\hat{k}\>.
	\end{equation*}
	The right-hand side of the above is
	\begin{align*}
		\<V_tV_s^*V_\mu\hat{h},V_\nu\hat{k}\>&=\<V_\nu^*V_tV_s^*V_\mu\hat{h},\hat{k}\>\\
		&=\<V_\nu^*V_tV_{(\mu-s)_-}^*V_{(\mu-s)_+}\hat{h},\hat{k}\>\\
		&=\<V_\nu^*V_tV_{(\mu-s)_+}V^*_{(\mu-s)_-}\hat{h},\hat{k}\>.
	\end{align*}
	Note that $t+(\mu-s)_+=(t+\mu-s)_+$ and $(\mu-s)_-=(t+\mu-s)_-$, hence we have
	\begin{align*}
		\<V_tV_s^*V_\mu\hat{h},V_\nu\hat{k}\>&=\<V^*_\nu V^*_{(t+\mu-s)_-}V_{(t+\mu-s)_+}\hat{h},\hat{k}\>\\
		&=\<V^*_{\nu+(t+\mu-s)_-}V_{(t+\mu-s)_+}\hat{h},\hat{k}\>\\
		&=\<V^*_{(t+\mu-\nu-s)_-}V_{(t+\mu-\nu-s)_+}\hat{h},\hat{k}\>,
	\end{align*}
	with the last equality coming from the fact that
	\begin{equation*}
		((t+\mu-s)_+-(t+\mu-s)_--\nu)_-=(t+\mu-s-\nu)_-
	\end{equation*}
	and
	\begin{equation*}
		((t+\mu-s)_+-(t+\mu-s)_--\nu)_+=(t+\mu-s-\nu)_+.
	\end{equation*}
	Hence
	\begin{align*}
		\<V_tV_s^*V_\mu\hat{h},V_\nu\hat{k}\>
		&=\<V^*_{(t+\mu-\nu-s)_-}V_{(t+\mu-\nu-s)_+}\hat{h},\hat{k}\>\\
		&=\<V_\nu^*V_s^*V_tV_\mu\hat{h},\hat{k}\>\\
		&=\<V_s^*V_tV_\mu\hat{h},V_\nu\hat{k}\>.
	\end{align*}
	Hence $V$ is Nica-covariant.
	
	To show that the dilation is unique we follow a standard argument. Suppose $V$ and $W$ are two minimal isometric Nica-covariant dilations of $T$ on
	$\Sc{K}_1$ and $\Sc{K}_2$ respectively. Take $h_1, h_2\in\Sc{H}$
	and $\nu,\mu\in\Sc{S}$. Then
	\begin{align*}
		\<V_\mu h_1,V_\nu h_2\>&=\<V_\nu^*V_\mu h_1,h_2\>\\
		&=\<V_{(\mu-\nu)_-}^*V_{(\mu-\nu)_+} h_1,h_2\>\\
		&=\<T_{\mu-\nu}h_1,h_2\>.
	\end{align*}
	Similarly $\<W_\mu h_1,W_\nu h_2\>=\<T_{\mu-\nu}h_1,h_2\>$. Thus the map $U:V_\nu h\mapsto W_\nu h$ extends to a unitary from $\Sc{K}_1$ to $\Sc{K}_2$ which 
	fixes $\Sc{H}$, and the two dilations $V$ and $W$ are unitarily equivalent.
\end{proof}

\section{Semicrossed product algebras}
Throughout let $\Sc{S}$ be the semigroup $\Sc{S}=\sum_{i=1}^{\oplus k}\Sc{S}_i$ where each $\Sc{S}_i$ is a countable subsemigroup of $\B{R}_+$ containing $0$. Further we suppose that $\Sc{S}$ is the positive cone of the group $\Sc{G}$ generated by $\Sc{S}$.

\begin{definition}
	Let $\Sc{A}$ be a unital operator algebra. If $\alpha=\{\alpha_{s}:\ s\in\Sc{S}\}$ is a family of completely isometric unital endomorphisms of $\Sc{A}$ forming an action 
	of $\Sc{S}$ on $\Sc{A}$ then we
	call the triple $(\Sc{A},\Sc{S},\alpha)$ a \emph{semigroup dynamical system}.
\end{definition}

\begin{definition}
	Let $(\Sc{A},\Sc{S},\alpha)$ be a semigroup dynamical system. An \emph{isometric (contractive) Nica-covariant representation} of $(\Sc{A},\Sc{S},\alpha)$ on a Hilbert space $\Sc{H}$ 
	consists of a pair $(\sigma, V)$ where $\sigma$ is a completely contractive representation $\sigma:\Sc{A}\rightarrow\Sc{B}(\Sc{H})$ and $V=\{V_s\}_{s\in\Sc{S}}$ is an isometric (contractive)
	Nica-covariant representation of $\Sc{S}$ on $\Sc{H}$ such that
	\begin{equation*}
		\sigma(A)V_s=V_s\sigma(\alpha_s(A))
	\end{equation*}
	for all $A\in\Sc{A}$ and $s\in\Sc{S}$.
\end{definition}

We will be interested in two nonself-adjoint semicrossed product algebras associated to a semigroup dynamical system $(\Sc{A},\Sc{S},\alpha)$. We define $\cross$ to be the universal algebra for all contractive Nica-covariant representations of $(\Sc{A},\Sc{S},\alpha)$ and $\isocross$ to be the universal algebra for all isometric Nica-covariant representations of $(\Sc{A},\Sc{S},\alpha)$.

The algebras $\Sc{A}\times_\alpha^{iso}\B{Z}_+$ were introduced by Kakariadis and Katsoulis \cite{kakkat} and have proven to be a more tractable class of algebras than $\Sc{A}\times_\alpha\B{Z}$. While in general one expects $\cross$ and $\isocross$ to be different there are times when the two algebras coincide. For example, when $\Sc{A}=\F{A}_n$ is the noncommutative disc algebra and $\Sc{S}=\B{Z}_+$ it follows
from \cite{DavKat3} that
	\begin{equation*}
		\F{A}_n\times_\alpha^{iso}\B{Z}_+\cong\F{A}_n\times_\alpha\B{Z}_+.
	\end{equation*}
Further examples of when the semicrossed product and the isometric semicrossed product are the same for the case $\Sc{S}=\B{Z}_+$ can be found in \cite[Section 12]{DavKat1}.
When $\Sc{A}$ is a unital $C^*$-algebra we will see (Corollary \ref{same alg}) that
	\begin{equation*}
		\isocross\cong\cross.
	\end{equation*}

Let $\Sc{P}(\Sc{A},\Sc{S})$ be the algebra of all formal polynomials $p$ of the form
	\begin{equation*}
		p=\sum_{i=1}^n\Sc{V}_{s_i}A_{s_i}
	\end{equation*}
where $s_1,\ldots,s_n$ are in $\Sc{S}$, with multiplication defined by $A\Sc{V}_s=\Sc{V}_s\alpha(A)$. If $(\sigma, T)$ is a contractive Nica-covariant representation of $(\Sc{A},\Sc{S},\alpha)$ then we can define a representation $\sigma\times T$ 
of $\Sc{P}(\Sc{A},\Sc{S})$ by
	\begin{equation*}
		(\sigma\times T)\left(\sum_{i=1}^n\Sc{V}_{s_i}A_{s_i}\right)=\sum_{i=1}^n T_{s_i}\sigma(A_{s_i}).
	\end{equation*}
We define two norms on $\Sc{P}(\Sc{A},\Sc{S})$ as follows. For $p\in\Sc{P}(\Sc{A},\Sc{S})$ let
	\begin{equation*}
		\|p\|=\sup_{\substack{(\sigma,T)\text{ contractive}\\\text{Nica-covariant}}}\Big\{(\sigma\times T)(p)
			\Big\}
	\end{equation*}
and
	\begin{equation*}
		\|p\|_{iso}=\sup_{\substack{(\sigma,V)\text{ isometric}\\\text{Nica-covariant}}}\Big\{(\sigma\times V)(p)
			\Big\}.
	\end{equation*}
We can realise our semicrossed product algebras as
	\begin{equation*}
		\cross=\overline{\Sc{P}(\Sc{A},\Sc{S})}^{\|\cdot\|}
	\end{equation*}
and
	\begin{equation*}
		\isocross=\overline{\Sc{P}(\Sc{A},\Sc{S})}^{\|\cdot\|_{iso}}.
	\end{equation*}

If $(\Sc{B},\Sc{G},\beta)$ is a dynamical system where $\beta$ is an action of the group $\Sc{G}$ on the $C^*$-algebra
$\Sc{B}$ by automorphisms there is an adjoint operation on $\Sc{P}(\Sc{B},\Sc{G})$ given by
$(\Sc{V}_gB)^*:=\Sc{V}_{-g}\beta_{g}^{-1}(B^*)$. If $(\pi,U)$ is covariant representation of $(\Sc{B},\Sc{G},\beta)$, then $\{U_s\}_{s\in\Sc{S}}$ is necessarily a family of
commuting unitaries, and hence $\{U_s\}_{s\in\Sc{S}}$ is automatically Nica-covariant.

\begin{example}\label{induced rep}
	Let $(\Sc{A},\Sc{S},\alpha)$ be a semigroup dynamical system. Let $\sigma$ be a completely contractive representation of $\Sc{A}$ on a Hilbert space $\Sc{H}$. Define
	a completely contractive representation $\tilde{\sigma}$ of $\Sc{A}$ on $\Sc{H}\otimes\ell^2(\Sc{S})$ by
	\begin{equation*}
		\tilde{\sigma}(A)(h_s)_{s\in\Sc{S}}=(\sigma(\alpha_s(A))h_s)_{s\in\Sc{S}}
	\end{equation*}
	for all $A\in\Sc{A}$ and $(h_s)_{s\in\Sc{S}}\in\Sc{H}\otimes\ell^2(\Sc{S})$.
	
	For each $s\in\Sc{S}$ define an operator $W_s$ on $\Sc{H}\otimes\ell^2(\Sc{S})$ by
	\begin{equation*}
		W_s(h)_t=(h)_{s+t},
	\end{equation*}
	where $h\in\Sc{H}$ and $(h)_s\in\Sc{H}\otimes\ell^2(\Sc{S})$ is the vector with $h$ in the $s^{th}$ position and $0$ everywhere else. Then $(\tilde{\sigma},W)$ is an isometric
	Nica-covariant representation of $(\Sc{A},\Sc{S},\alpha)$.
	
	Note that in the case where each $\alpha_s$ is an automorphism on $\Sc{A}$ then we can extend this idea to give a Nica-covariant representation $(\hat{\sigma}, U)$ on
	$\Sc{H}\otimes\ell^2(\Sc{G})$ where each $U_s$ is unitary.
\end{example}

\begin{definition}
	The isometric Nica-covariant representation $(\tilde{\sigma},W)$ constructed above is called an \emph{induced representation} of $(\Sc{A},\Sc{S},\alpha)$.
\end{definition}

\subsection{Dilations of Nica-covariant representations}

We now consider some dilation results for Nica-covariant representations of a semigroup dynamical system $(\Sc{A},\Sc{S},\alpha)$ in the case when $\Sc{A}$ is a $C^*$-algebra.

In the case that $\Sc{S}=\B{Z}_+^k$ the following theorem is a special case of a theorem of Solel's \cite[Theorem 3.1]{Solel} which deals with representations of product systems of $C^*$-correspondences. The result has also been shown by Ling and Muhly \cite{LingMuhly} for the case $\Sc{S}=\B{Z}_+^k$ and $\alpha$ is an action on $\Sc{A}$ by automorphisms. 

\begin{theorem}\label{min iso dil}
	Let $\Sc{S}=\sum_{i\in I}^{\oplus}\Sc{S}_i$ where each $\Sc{S}_i$ is a countable subsemigroup of $\B{R}_+$ containing $0$ and
	let $(\Sc{A},\Sc{S},\alpha)$ be a semigroup dynamical system where $\Sc{A}$ is a unital $C^*$-algebra. Let $(\sigma,T)$ be a contractive Nica-covariant representation of $(\Sc{A},\Sc{S},\alpha)$ on $\Sc{H}$.
	Then there is an isometric Nica-covariant representation $(\pi,V)$ of $(\Sc{A},\Sc{S},\alpha)$ on $\Sc{K}\supseteq\Sc{H}$ such that
	\begin{enumerate}
		\item $\pi(A)|_\Sc{H}=\sigma(A)$ for all $A\in\Sc{A}$
		\item $P_\Sc{H}V_s|_\Sc{H}=T_s$ for all $s\in\Sc{S}$.
	\end{enumerate}
	Further $\Sc{K}$ is minimal in the sense that $\Sc{K}=\bigvee_{s\in\Sc{S}}V_s\Sc{H}$.
\end{theorem}

\begin{proof}
	Let $\Sc{K}_0$, $\Sc{K}$ and $\Sc{N}$ be as in the proof of Theorem \ref{dilation is Nica-cov}. For each $A\in\Sc{A}$ we define $\pi_0(A)$ on $\Sc{K}_0$ by
	\begin{equation*}
		(\pi_0(A)f)(s)=\sigma(\alpha_s(A))f(s),
	\end{equation*}
	for each $f\in\Sc{K}_0$ and $s\in\Sc{S}$. Note that, for $A\in\Sc{A}$ and $t,s\in\Sc{S}$ we have
	\begin{align*}
		T_{t-s}\sigma(\alpha_t(A))&=T_{(t-s)_+}T_{(t-s)_-}^*\sigma(\alpha_t(A))\\
		&=T_{(t-s)_+}\sigma(\alpha_{t+(t-s)_-}(A))T_{(t-s)_-}^*\\
		&=\sigma(\alpha_{t+(t-s)_- -(t-s)_+}(A))T_{(t-s)_-}^*T_{(t-s)_+}\\
		&=\sigma(\alpha_s(A))T_{t-s}.
	\end{align*}
	
	It follows that, if $f\in\Sc{N}$  and $g\in\Sc{K}_0$ then for each $A\in\Sc{A}$,
	\begin{align*}
		\<\pi_0(A)f,g\>&=\sum_{s,t}\<T_{t-s}\sigma(\alpha_t(A))f(t),g(s)\>\\
		&=\sum_{s,t}\<T_{t-s}f(t),\sigma(\alpha_s(A^*))g(s)\>=0,
	\end{align*}
	we thus can extend $\pi_0$ to a representation $\pi$
	\begin{equation*}
		\pi:\Sc{A}\rightarrow\Sc{B}(\Sc{K}).
	\end{equation*}
	It is easy to check that $(\pi,V)$ form a Nica-covariant representation with the desired properties.
\end{proof}

\begin{remark}
	In the case when $\Sc{S}=\sum_{i\in I}^{\oplus}\Sc{S}_i$ where each $\Sc{S}_i$ is a subsemigroup of $\B{R}_+$ containing $0$ and each $\Sc{S}_i$ has the extra		condition of being commensurable then the statement of Theorem \ref{min iso dil} is a special case of \cite[Theorem 4.2]{Shal2}. However, in the proof there, the only place where the 		commensurable condition is
	used is in ensuring that contractive Nica-covariant representation of $\Sc{S}$ has minimal Nica-covariant isometric dilation. As Theorem \ref{dilation exists} and Theorem
	\ref{dilation is Nica-cov} provide the existence of minimal Nica-covariant isometric dilations in the case when each $\Sc{S}_i$ is not necessarily commensurable the proof
	given in \cite{Shal2} provides an alternate proof of Theorem \ref{min iso dil}.
\end{remark}

\begin{corollary}\label{same alg}
	Let $\Sc{S}=\sum_{i=1}^{\oplus k}\Sc{S}_i$ where each $\Sc{S}_i$ is a countable subsemigroup of $\B{R}_+$ containing $0$ and
	let $(\Sc{A},\Sc{S},\alpha)$ be a semigroup dynamical system where $\Sc{A}$ is a unital $C^*$-algebra.
	Then the norms $\|\cdot\|$ and $\|\cdot\|_{iso}$ on $\Sc{P}(\Sc{A},\Sc{S})$ are the same. Hence
		\begin{equation*}
			\isocross=\cross.
		\end{equation*}
\end{corollary}

\begin{proof}
	Take any $p\in\Sc{P}(\Sc{A},\Sc{S})$.
	Since an isometric Nica-covariant representation is itself contractive it follows that $\|p\|_{iso}\leq\|p\|$. 
	Now take a contractive Nica-covariant representation $(\sigma,T)$ on a Hilbert space $\Sc{H}$. Let $(\pi, V)$
	be the minimal isometric Nica-covariant
	dilation of $(\sigma,T)$. Then
	\begin{equation*}
		\|(\sigma\times T)(p)\|=\|P_{\Sc{H}}(\pi\times V)(p)P_{\Sc{H}}\|\leq\|(\pi\times V)(p)\|.
	\end{equation*}
	Hence $\|p\|\leq\|p\|_{iso}$.
\end{proof}

\begin{remark}\label{tensor alg}
	Let $(\Sc{A},\Sc{S},\alpha)$ be a semigroup dynamical system. If $\Sc{A}$ is a $C^*$-algebra then 
	$(\Sc{A},\Sc{S},\alpha)$ can be used to describe a product system of $C^*$-correspondences over $\Sc{S}$. Fowler
	constructs a concrete $C^*$-algebra which is universal for Nica-covariant completely contractive representations of 		this product system \cite{Fowler}. It was observed by Solel \cite{Solel} that the nonself-adjoint Banach algebra
	formed by the left regular representation of the product system is universal for Nica-covariant completely
	contractive representations (while Solel was working in $\B{Z}_+^k$ the same reasoning works for countable $\Sc{S}$). 		Thus $\cross$ can also be realised as the concrete tensor algebra in the sense of Solel, see
	\cite[Corollary 3.17]{Solel}.
	
	Further, if $\sigma$ is a faithful representation of $\Sc{A}$ it follows that the induced representation $(\tilde{\sigma},W)$ is a completely isometric representation of $\cross$.
\end{remark}

The following theorem can be proved by a standard argument in dynamical systems using direct limits of $C^*$-algebras. As stated below, the result is a special case of 
\cite[Theorem 2.1]{laca} and \cite[Section 2]{murph}.

\begin{theorem}\label{min auto dil}
	Let $(\Sc{A},\Sc{S},\alpha)$ be a semigroup dynamical system where $\Sc{A}$ is a $C^*$-algebra and each $\alpha_s$ is injective. Then there exists a
	$C^*$-dynamical system $(\Sc{B},\Sc{G},\beta)$
	where each $\beta_s$ is an automorphism, 
	unique up to isomorphism, together with an embedding $i:\Sc{A}\rightarrow\Sc{B}$ such that
	\begin{enumerate}
		\item $\beta_s\circ i=i\circ\alpha_s$, i.e. $\beta$ dilates $\alpha$
		\item $\bigcup_{s\in\Sc{S}}\beta_s^{-1}(i(\Sc{A}))$ is dense in $\Sc{B}$, i.e. $\Sc{B}$ is minimal.
	\end{enumerate}
\end{theorem}

\begin{definition}
	Let $(\Sc{A},\Sc{S},\alpha)$ and $(\Sc{B},\Sc{G},\beta)$ be as in Theorem \ref{min auto dil}, then we call $(\Sc{B},\Sc{G},\beta)$ the \emph{minimal automorphic dilation} of
	$(\Sc{A},\Sc{S},\alpha)$.
\end{definition}

The minimal automorphic dilation of a dynamical system is frequently utilised in the literature. Group crossed product $C^*$-algebras have a long history and are well understood objects. Thus it is beneficial if one can relate a semicrossed algebra to a crossed product algebra, often the crossed product algebra of the minimal automorphic dilation. We will see in Theorem \ref{c* env} that the minimal automorphic dilation plays an important role when calculating the $C^*$-envelope of crossed product algebras.  First we will show now that $\cross$ sits nicely inside
$\fullcross$. In the case where $\Sc{S}=\B{Z}_+$ the following has been shown by Kakariadis and Katsoulis \cite{kakkat} and Peters \cite{peters1}.

\begin{theorem}\label{subalgebra}
	Let $(\Sc{A},\Sc{S},\alpha)$ be a semigroup dynamical system where $\Sc{A}$ is a $C^*$-algebra and each $\alpha_s$ is injective. Let $(\Sc{B},\Sc{G},\beta)$ be the minimal
	automorphic dilation of $(\Sc{A},\Sc{S},\alpha)$. The $\cross$ is completely isometrically isomorphic to a subalgebra of $\fullcross$. 

	Further, $\isocross$ generates $\fullcross$ as a $C^*$-algebra.
\end{theorem}

\begin{proof}
	Let $\sigma$ be a faithful representation of $\Sc{A}$ on $\Sc{H}$.
	Then the induced representation $\tilde{\sigma}\times W$ is a completely isometric representation of $\cross$, by Remark \ref{tensor alg}. We will 
	embed this completely isometric copy of $\cross$ into a completely isometric representation of $\fullcross$ by suitably dilating the representation $(\tilde{\sigma},W)$.
	
	Let $i$ be the embedding of $\Sc{A}$ into $\Sc{B}$ as in Theorem \ref{min auto dil}. The representation $\sigma$ also defines a faithful representation of $i(\Sc{A})$,
	which we will also denote by $\sigma$. We can 
	thus find a representation $\pi$ of $\Sc{B}$ on $\Sc{K}\supseteq\Sc{H}$ such that $\pi(A)|_{\Sc{H}}=\sigma(i(A))$ for all $A\in\Sc{A}$, see e.g. \cite[Proposition 4.1.8]{ped}.
	We thus have an induced representation $\hat{\pi}\times U$ of $\fullcross$. Restricting $\pi$ to $\Sc{A}$ we see that $(\hat{\pi}\circ i)\times U$ is a completely isometric
	representation of $\cross$, since $\tilde{\sigma}\times W$ is. Further note that  $\hat{\pi}$ is faithful on $\bigcup_{s\in\Sc{S}}\beta_s^{-1}(\Sc{A})$. By the construction
	of $\Sc{B}$, $\hat{\pi}$ is also faithful representation of $\Sc{B}$. Now, by \cite[Theorem 7.7.5]{ped}, $\tilde{\sigma}\times W$ is a faithful representation of $\fullcross$. Hence
	$\cross$ sits completely isometrically inside $\fullcross$.
	
	That $\cross$ generates $\fullcross$ as a $C^*$-algebra follows immediately after considering the algebra
	$\Alg\{\Sc{P}(\Sc{A},\Sc{S}),(\Sc{P}(\Sc{A},\Sc{S}))^*\}$ inside $\Sc{P}(\Sc{B},\Sc{G})$. 
\end{proof}

\subsection{$C^*$-Envelopes}
Our goal in this subsection is to calculate the $C^*$-envelope of $\isocross$ in the case when $\alpha$ is a family of completely isometric automorphisms on a unital operator algebra $\Sc{A}$. 

If $\Sc{C}$ is a $C^*$-algebra which completely isometrically contains $\Sc{A}$ such that $\Sc{C}=C^*(\Sc{A})$ then we call
$\Sc{C}$ a \emph{$C^*$-cover of
$\Sc{A}$}. If $\Sc{A}$ is a $C^*$-algebra, Theorem \ref{subalgebra} says that $\fullcross$ is a $C^*$-cover of $\isocross$ when $(\Sc{B},\Sc{G},\beta)$ is the minimal automorphic dilation of $(\Sc{A},\Sc{S},\alpha)$. 

\begin{definition}
	Let $\Sc{A}$ be an operator algebra and let $\Sc{C}$ be a $C^*$-cover of $\Sc{A}$. Let $\alpha$ define an action of
	$\Sc{S}$ on $\Sc{C}$ by faithful $^*$-endomorphisms which leave $\Sc{A}$ invariant. We define the \emph{relative semicrossed 	 product} $\relcross$ to
	be the subalgebra of $\Sc{C}\sideset{_N}{_\alpha}{\mathop{\times}}\Sc{S}$ generated by the natural copy of $\Sc{A}$ inside
	$\Sc{C}\sideset{_N}{_\alpha}{\mathop{\times}}\Sc{S}$ and the universal isometries $\{\Sc{V}_s\}_{s\in\Sc{S}}$.
\end{definition} 

The idea of a relative semicrossed product was introduced by Kakariadis and Katsoulis \cite{kakkat} when studying semicrossed products by the semigroup $\B{Z}_+$. The key idea is to realise the universal algebra 
$\isocross$ as a relative semicrossed algebra. This allows a concrete place in which to try and discover the $C^*$-envelope.

The proof of the following proposition follows the same reasoning as the proof of \cite[Proposition 2.3]{kakkat}. It is an application of Dritschel and McCullough's \cite{DritMcCul} result that any representation can be dilated to a maximal representation and Muhly and Solel's \cite{muhlysolel1} result that any maximal representation extends to a $^*$-representation of any $C^*$-cover.

It is also important to note that if $\alpha$ is an action of $\Sc{S}$ on an operator algebra $\Sc{A}$ by completely isometric automorphisms  which extend to completely isometric automorphisms of a $C^*$-cover $\Sc{C}$ of
$\Sc{A}$, then each $\alpha_s$ necessarily leaves the Shilov boundary $\Sc{J}$ of $\Sc{A}$ in $\Sc{C}$ invariant, see e.g. 
\cite[Proposition 10.6]{DavKat1}. We will write $\{\dot{\alpha_s}\}_{s\in\Sc{S}}$ for the automorphisms on $\Sc{A}/\Sc{J}$ induced by the
automorphisms $\{\alpha_s\}_{s\in\Sc{S}}$ on $\Sc{A}$.

\begin{proposition}
	Let $\Sc{A}$ be an operator algebra and let $\Sc{C}$ be a $C^*$-cover of $\Sc{A}$. Let $\alpha$ be an action of
	$\Sc{S}$ on $\Sc{C}$ by automorphisms that restrict to automorphisms of $\Sc{A}$. Let $\Sc{J}$ be the Shilov 		boundary of $\Sc{A}$ in
	$\Sc{C}$. Then the relative semicrossed products $\relcross$ and
	$\Sc{A}/\Sc{J}\sideset{_N}{_{\Sc{C}/\Sc{J},\dot{\alpha}}}{\mathop{\times}}\Sc{S}$ are completely isometrically
	isomorphic.
\end{proposition}

Let $(\Sc{C},\Sc{S},\alpha)$ be a semigroup dynamical system where $\Sc{C}$ is a $C^*$-algebra and each $\alpha_s$ is an automorphism on $\Sc{C}$. Then it is immediate that the minimal automorphic dilation of $(\Sc{C},\Sc{S},\alpha)$ is simply
$(\Sc{C},\Sc{G},\alpha)$. If we view $\Sc{G}$ as being a discrete group then $\Sc{G}$ has a compact dual $\hat{\Sc{G}}$. Recall that for every character $\gamma$ in $\hat{\Sc{G}}$ we can define an automorphism $\tau_\gamma$ on $\Sc{P}(\Sc{C},\Sc{G})$ by
	\begin{equation*}
		\tau_\gamma\left(\sum_{i=1}^n\Sc{V}_{s_i}A_{s_i}\right)=\sum_{i=1}^n\gamma(s_i)\Sc{V}_{s_i}A_{s_i}.
	\end{equation*} 
The automorphism $\tau_\gamma$ extends to an automorphism of $\Sc{C}\times_\alpha\Sc{G}$ with $\Sc{C}$ as its fixed-point set
\cite[Proposition 7.8.3.]{ped}. We call $\tau_\gamma$ a \emph{gauge automorphism}. The gauge automorphisms restrict to automorphisms of
$\Sc{C}\sideset{_N}{_\alpha}{\mathop{\times}}\Sc{S}$.

\begin{lemma}\label{c* env lemma}
	Let $\Sc{A}$ be a unital operator algebra. Let $\Sc{C}$ be a $C^*$-cover of $\Sc{A}$ and let $\Sc{J}$ be the Shilov
	boundary of $\Sc{A}$ in $\Sc{C}$. Let $\alpha$ be an action of $\Sc{S}$ on $\Sc{C}$ by automorphisms which restrict 		to completely isometric automorphisms of $\Sc{A}$. Then
	\begin{equation*}
		C^*_{env}(\relcross)\cong
		C^*_{env}(\Sc{A})\times_{\dot{\alpha}}\Sc{G}.
	\end{equation*}
\end{lemma}

\begin{proof}
	By the preceding proposition it suffices to show that
	\begin{equation*}
		C^*_{env}(\Sc{A}/\Sc{J}\sideset{_N}{_{\Sc{C}/\Sc{J},\dot{\alpha}}}{\mathop{\times}}\Sc{S})\cong
		\Sc{C}/\Sc{J}\times_{\dot{\alpha}}\Sc{G}.
	\end{equation*}

	The algebra $\Sc{A}/\Sc{J}\sideset{_N}{_{\Sc{C}/\Sc{J},\dot{\alpha}}}{\mathop{\times}}\Sc{S}$ embeds completely
	isometrically into $\Sc{C}/\Sc{J}\times_{\dot{\alpha}}\Sc{G}$ and 
	generates it as a $C^*$-algebra. Let $\Sc{I}$ be the Shilov boundary of
	$\Sc{A}/\Sc{J}\sideset{_N}{_{\Sc{C}/\Sc{J},\dot{\alpha}}}{\mathop{\times}}\Sc{S}$ in 
	$\Sc{C}/\Sc{J}\times_{\dot{\alpha}}\Sc{G}$. Suppose that $\Sc{I}\not=\{0\}$.

	The ideal $\Sc{I}$ is invariant under automorphisms of 
	$\Sc{C}/\Sc{J}\times_{\dot{\alpha}}\Sc{G}$ and hence by the gauge automorphisms of
	$\Sc{C}/\Sc{J}\times_{\dot{\alpha}}\Sc{G}$. Therefore $\Sc{I}$ has non-trivial intersection with the fixed points of
	the gauge automorphisms, i.e. $\Sc{I}\cap\Sc{C}/\Sc{J}\not=\{0\}$. But $\Sc{I}\cap\Sc{C}/\Sc{J}$ is a boundary ideal 		for $\Sc{A}$ in $\Sc{C}/\Sc{J}$. Hence $\Sc{I}=\{0\}$. This proves the result.
\end{proof}

We can now prove the main result of this section. This theorem generalises the result of Kakariadis and Katsoulis \cite{kakkat} from the semigroup $\B{Z}_+$ to our more general
semigroups $\Sc{S}=\sum_{i=1}^{\oplus k}\Sc{S}_i$. From another viewpoint, in the case when
$\Sc{A}$ is a $C^*$-algebra and $\isocross\cong\cross$ we have that the $C^*$-envelope of an associated tensor algebra is a crossed product algebra, by Remark \ref{tensor alg} and Corollary \ref{same alg}. This was shown for abelian $C^*$-algebras by Duncan and Peters \cite{DuncanPeters}.

By \cite[Proposition 10.1]{DavKat1} the group $\text{Aut}(\Sc{A})$ of completely isometric automorphisms on the unital operator algebra $\Sc{A}$ is isomorphic to the group of
completely isometric automorphisms on $C^*_{env}(\Sc{A})$ which leave $\Sc{A}$ invariant. Thus, if $\{\alpha_s\}_{s\in\Sc{S}}$ a family of completely isometric automorphisms defining an action of $\Sc{S}$ on $\Sc{A}$, then they can be extended to a family completely isometric automorphisms defining an action of $\Sc{S}$ on $C^*_{env}(\Sc{A})$. 

\begin{theorem}\label{c* env}
	Let $\Sc{A}$ be a unital operator algebra. Let $\alpha$ be an action of $\Sc{S}$ on $\Sc{A}$ by completely isometric
	automorphisms. Denote also by $\alpha$ the extension of this action to $C^*_{env}(\Sc{A})$.
	Then 
	\begin{equation*}
		C^*_{env}(\isocross)\cong C^*_{env}(\Sc{A})\times_\alpha \Sc{G}.
	\end{equation*}
\end{theorem}

\begin{proof}
	We will show that $\isocross$ is isomorphic to a relative semicrossed product. The result will then follow by Lemma \ref{c* env lemma}.
	
	Let $\{\Sc{V}_s\}_{s\in\Sc{S}}$ be the universal isometries in $\isocross$ acting on a Hilbert space $\Sc{H}$. For each $s\in\Sc{S}$  let 
	$\Sc{H}_s=\Sc{H}$ and define maps $\Sc{V}^{s,t}$ when $s\leq t$
	\begin{equation*}
			\Sc{V}^{s,t}:\Sc{H}_s\rightarrow\Sc{H}_t
	\end{equation*}
	by $\Sc{V}^{s,t}=\Sc{V}_{t-s}$. Let $\Sc{K}$ be the Hilbert space inductive limit of the directed system
	$(\Sc{H}_s)_{s\in\Sc{S}}$.
	
	For each $A\in\Sc{A}$ the commutative diagram
	\begin{equation*}
	\begin{CD}
		\Sc{H} @> \Sc{V}_s>> \Sc{H}\\
		@V A VV		@V \alpha_{s}^{-1}(A) VV\\
		\Sc{H} @> \Sc{V}_s>> \Sc{H}
	\end{CD}
	\end{equation*}
	defines an operator $\pi(\Sc{A})$ on $\Sc{K}$. Thus we have a completely isometric representation $\pi:\Sc{A}\rightarrow\Sc{B}(\Sc{K})$.
	
	Now for each $s,t\in\Sc{S}$ define operator $U^s_t:\Sc{H}_s\rightarrow\Sc{H}_s$ by $U^s_t=\Sc{V}_t$. Passing to the direct limit
	we get a family of commuting unitaries $\{U_s\}_{s\in\Sc{S}}$ on $\Sc{K}$ satisfying
	\begin{equation*}
		\pi(A)U_s=U_s\pi(\alpha_s(A)).
	\end{equation*}
	The unitaries $\{U_s\}_{s\in\Sc{S}}$ thus define $^*$-automorphisms of $\Sc{C}:=C^*(\pi(A))$ extending $\alpha$. Thus
	\begin{equation*}
		\isocross\cong\Sc{A}\sideset{_N}{_{\Sc{C},{\alpha}}}{\mathop{\times}}\Sc{S}.
	\end{equation*}
	The result now follows by Lemma \ref{c* env lemma}.
\end{proof}

  \subsection{Acknowledgements}The author would like to thank his advisor, Ken Davidson, for his advice and support.

\end{document}